\documentclass[12pt,a4paper]{amsart}
\usepackage{amsmath, amssymb, amsfonts, amsthm}
\usepackage{graphicx}
\usepackage{color}
\usepackage[colorlinks=true]{hyperref}
\hypersetup{urlcolor=blue, citecolor=blue, draft=false}
\allowdisplaybreaks
\usepackage{amsmath,amssymb,amsthm,amsfonts,mathrsfs,enumerate,sansmath,mathtools,}

\newtheorem{theorem}{Theorem}[section]
\newtheorem{lemma}{Lemma}[section]
\newtheorem{corollary}{Corollary}[section]
\theoremstyle{remark}

\theoremstyle{remark}

\begin{document}
\title[CERTAIN GEOMETRIC PROPERTIES OF CLOSE-TO-CONVEX HARMONIC MAPPINGS]{CERTAIN GEOMETRIC PROPERTIES OF CLOSE-TO-CONVEX HARMONIC MAPPINGS}

\author[Rajbala and Jugal K. Prajapat]{Rajbala and Jugal K. Prajapat}

\address{Department of Mathematics, Central University of Rajasthan, Bandarsindri, Kishangarh-305817, Dist.-Ajmer, Rajasthan, India}
\email{rajbalachoudary9@gmail.com, jkprajapat@gmail.com}

\date{}
\begin{abstract}
In this article, we introduce a new family of sense preserving harmonic mappings  $f=h+\overline{g}$ in the open unit  disk and prove that functions in this family are close-to-convex. We give some basic properties such as coefficient bounds, growth estimates, convolution and determine the radius of convexity for the functions belonging to this family. In addition, we construct certain harmonic univalent polynomials belonging to this family.
\end{abstract}

\maketitle
\section{Introduction}
\setcounter{equation}{0}

Let $\mathcal{H}$ denote  the class of complex valued harmonic functions $f$ in $\mathbb{D}$ normalized by $f(0)=f_z(0)-1=0.$ Each such function $f$ can be expressed uniquely as $f=h+\overline{g},$ where $h$ and $g$ have the following power series representations:
\begin{equation}\label{intro1}
h(z) = z + \sum_{n=2}^{\infty} a_nz^n  \quad {\rm and } \quad g(z)=\sum_{n=1}^{\infty} b_nz^n.
\end{equation} 
A result of Lewy \cite{lewy}, shows that $f\in \mathcal{H}$ is locally univalent in $\mathbb{D}$ if and only if $J_f(z) = |f_z(z)|^2-|f_{\overline{z}}(z)|^2$  is non-zero in $\mathbb{D},$ and is sense preserving  if $J_f(z)>0 \;(z\in \mathbb{D}),$ or equivalently, if the dilatation $w=g' /h'$ is analytic and satisfies $|w|<1$ in $\mathbb{D}.$ Observe that, the class $\mathcal{H}$ reduces to the class $\mathcal{A}$ of normalized analytic functions if the co-analytic part is zero. Let $\mathcal{S}_\mathcal{H}$ be the subclass of $\mathcal{H}$ consisting  of univalent  and sense-preserving harmonic mappings in $\mathbb{D}.$ The classical family $\mathcal{S}$ of normalized analytic univalent functions is subclass of $\mathcal{S}_{\mathcal{H}}$ as $\mathcal{S}=\{f=h+\overline{g}\in \mathcal{S}_{\mathcal{H}}:g\equiv 0 \quad {\rm in} \quad \mathbb{D}  \}.$ Also, we denote by $\mathcal{H}^0=\left\lbrace  f\in \mathcal{H}: f_{\overline{z}}(0)=0 \right\rbrace $ and  $\mathcal{S}_\mathcal{H}^0=\left\lbrace  f\in \mathcal{S} _\mathcal{H}: f_{\overline{z}}(0)=0 \right\rbrace.$ It is well known that the class $\mathcal{S}_\mathcal{H}^0$ is compact and normal, whereas the class $\mathcal{S}_\mathcal{H}$ is normal but not compact.  In 1984, Clunie and Sheil-Small \cite{clunie} investigated the class  $\mathcal{S}_\mathcal{H},$ together with some of its geometric subclasses. 

A function $h \in \mathcal{A}$ is called close-to-convex in $\mathbb{D},$ if the  complement of $h(\mathbb{D})$ can be written as the union of non-intersecting half lines. Let $\mathcal{C}$ denote the class of close-to-convex functions in $\mathbb{D}$. By $\mathcal{C}_{\mathcal{H}},$ we denote the class of close-to-convex harmonic mappings $f=h+\overline{g}$ for which $f(\mathbb{D})$ is close-to-convex in $\mathbb{D}.$ An analytic function $h \in \mathcal{A}$ is close-to-convex in $\mathbb{D},$ if there exists an convex function $\phi$ (not necessarily normalized) in $\mathbb{D}$ such that
$$\Re\left( \dfrac{h'(z)}{\phi'(z)}\right)> 0\qquad (z\in \mathbb{D}).$$ 

If $\phi(z)=z,$ then functions $h\in \mathcal{A}$ which satisfy $\Re( h'(z))>0,$ are close-to-convex in $\mathbb{D}.$ A function $h \in \mathcal{A}$ is said to be close-to-convex function of order $\beta \;(0 \leq \beta <1)$, if it satisfies $\Re(h'(z))>\beta \;(z\in \mathbb{D})$. Let $\mathcal{W}(\alpha,\beta)$ denote a class of functions $h \in \mathcal{A}$ such that $\Re(h'(z)+\alpha zh''(z))>\beta \;\; (\alpha \geq 0, 0 \leq \beta <1).$ The class $\mathcal{W}(\alpha, \beta)$ was studied by Gao and Zohu \cite{AAaa} for $\beta<1$ and $\alpha>0.$ They determined the extreme points of $\mathcal{W}(\alpha, \beta)$ and obtained a number $\beta(\alpha)$ such that $\mathcal{W}(\alpha, \beta)\subset \mathcal{S}^*$ for fixed $\alpha \in [1,\infty).$ The class $\mathcal{W}(\alpha,\beta)$ is  generalization of class $\mathcal{W}(\alpha) \equiv  \mathcal{W}(\alpha,0)$, which was studied by Chichra \cite{chichra77}. In \cite{singh89}, Singh and Singh proved that functions in $\mathcal{W}(1,0)$ are starlike in $\mathbb{D}.$

A harmonic function $f\in \mathcal{H}$ is said to be convex in $\mathbb{D}$, if $f(\mathbb{D})$ is convex in $\mathbb{D}$. We denote by $\mathcal{K}_{\mathcal{H}}\,$ the class of functions in $\mathcal{H}$ which are convex in $\mathbb{D}.$ A sense preserving harmonic mapping $f=h+\overline{g} \in \mathcal{H}$ is known to be convex in $\mathbb{D},$ if $\frac{\partial}{\partial \theta}\left(arg\,\left(\frac{\partial}{\partial \theta}f(re^{i \theta})\right)\right)>0$ for all $z=re^{i \theta}\in \mathbb{D}/\{0\}.$ Hence, $f=h+\overline{g}\in \mathcal{H}$ is convex in $\mathbb{D},$ if $f(z)\neq 0$ for all $z\in \mathbb{D}/\{0\}$ and condition
$$\Re\left\{\dfrac{z(h'(z)+zh''(z))+\overline{z(g'(z)+zg''(z))}}{zh'(z)-zg'(z)} \right\}>0$$
is satisfied for all $z \in \mathbb{D}/\{0\}.$  

Let $h\in \mathcal{S}$ be given by $h(z)=\sum_{n=0}^{\infty}a_n z^n.$ Then the $n^{th}$ partial sum (or section) of $h(z)$ is defined by 
$$s_n(h)=\sum_{k=0}^{n}a_kz^k \quad {\rm for}\quad n\in \mathbb{N},$$
where $a_0=0$ and $a_1=1.$ One of the classical results of Szeg\"{o} \cite{szego28} shows that if $h \in \mathcal{S},$ then the partial sum $s_n(h)(z)=\sum_{k=0}^{n}a_kz^k$ is univalent in disk $|z|<1/4$ for all $n\geq 2,$ and number $1/4$ can not be replaced by larger one. In \cite{AAA}, Robertson proved that $n^{th}$ partial sum of the Koebe function $k(z)=z/(1-z)^2$ is starlike in the disk $|z|<1-3n^{-1} \log n \quad (n\geq 5),$ and number $3$ can not be replaced by smaller constant. It is known by a result \cite[p. 256, 273]{Durenp}, that $s_n(h)$ is convex, starlike, or close-to-convex in the disk $|z|<1-3n^{-1} \log n\quad (n\geq 5),$ whenever $h$ is convex, starlike or close-to-convex in $\mathbb{D}.$ The largest radius $r_n$ of univalence of $s_n(h) \,(h \in \mathcal{S})$ is not yet known. However, Jenkins \cite{jenkins} (see also \cite[Section 8.2]{Durenp}) observed that $r_n \geq 1-(4+\varepsilon) n^{-1} \,\mbox{log} \,n$ for each $\epsilon\,(|\epsilon|=1)$ and for large $n$. There exists a considerable amount of results in the literature for partial sums of functions in the class $\mathcal{S}$ and some of its geometric subclasses. 

Analogously in the harmonic case, the $(p,q)$-th partial sum of a harmonic mapping $f=h+\overline{g}\in \mathcal{H}$ is defined by 
$$ s_{p,q}(f)=s_p(h)+\overline{s_q(h)},$$
where $s_p(h)=\sum_{k=1}^{p}a_kz^k$ and $s_q(g)=\sum_{k=1}^{q}b_kz^k$, $p,q\geq 1$ with $a_1=1, \, p\geq 1$ and $q\geq2$. In \cite{li13}, Li and Ponnusamy studied the radius of univalency of partial sums of functions in the class $\mathcal{P}_\mathcal{H}^0=\{f=h+\bar{g}\in \mathcal{H}^0: \,\, \Re(h'(z))>|g'(z)| \;(z\in \mathbb{D})\}.$ Further, in \cite{s13}, Li and Ponnusamy studied partial sums of functions in the class $\mathcal{P}_\mathcal{H}^0(\alpha)=\{f=h+\overline{g} \in \mathcal{H}^0: \Re(h'(z)-\alpha) > |g'(z)| \; (\alpha<1, \;z \in \mathbb{D})\}$. Recently, Ghosh and Vasudevarao \cite{ghosh17} studied a class of harmonic mappings $\mathcal{W}_\mathcal{H}^0(\alpha)=\{f=h+\bar{g} \in \mathcal{H}^0 : \,\Re(h'(z)+\alpha zh''(z))> |g'(z)+\alpha zg''(z)|\;  (z \in \mathbb{D})\}$ and gave some results concerning growth, convolution and convex combination for the members of the class $\mathcal{W}_\mathcal{H}^0(\alpha).$ 

For two analytic functions  $\psi_1(z)=\sum_{n=0}^{\infty}a_nz^n$ and $\psi_2(z)=\sum_{n=0}^{\infty}b_nz^n,$ the convolution (or Hardamard product) is defined by $ \left( \psi_1\ast \psi_2 \right) (z)= \sum_{n=0}^{\infty} a_nb_nz^n \; (z\in\mathbb{D}).$ Analogously in the harmonic case, for two harmonic mappings $f_1=h_1+\overline{g_1}$ and $f_2=h_2+\overline{g_2}$ in $\mathcal{H}$ with the power series of the form
$$f_1(z)=z+\sum_{n=2}^{\infty}a_nz^n+\overline{\sum_{n=1}^{\infty}b_n\, z^n} \quad {\rm and} \quad f_2(z)=z+\sum_{n=2}^{\infty}A_nz^n+\overline{\sum_{n=1}^{\infty}B_n\, z^n},$$
we define the harmonic convolution as follows:
 $$\; f_1 \ast f_2 = h_1\ast h_2+\overline{g_1\ast g_2}=z+\sum_{n=2}^{\infty}a_n A_n z^n+\overline{\sum_{n=1}^{\infty}b_n\,B_n\, z^n}.$$ 
Clearly, the class $\mathcal{H}$ is closed under the convolution, {\it i.e.} $\mathcal{H}\ast \mathcal{H}\subset \mathcal{H}.$ In the case of conformal mappings, the literature about convolution theory is exhaustive. Unfortunately, most of these results do not necessarily carry over to the class of univalent harmonic mappings in $\mathbb{D}.$  We refer \cite{droff1, kumar, ELiu}, for more information about convolution of harmonic mappings. 

\medskip
We now define a new class of close-to-convex harmonic mappings as follows:

\medskip
\noindent
{\bf Definition 1.1.} For $\alpha \geq 0$ and $0 \leq \beta <1$, let $\mathcal{W}_\mathcal{H}^0(\alpha,\beta)$ denote the class of harmonic mappings $f=h+\overline{g}$, which is defined by
\begin{equation*}
\mathcal{W}_\mathcal{H}^0(\alpha,\beta)=\{f=h+\bar{g} \in \mathcal{H}^0: \;\Re(h'(z)+\alpha zh''(z)-\beta)> |g'(z)+\alpha zg''(z)|\quad (z \in \mathbb{D}) \}.
\end{equation*}

We observe that, the class $\mathcal{W}_{\mathcal{H}}^0(\alpha,\beta)$ generalizes several previously studied classes of harmonic mappings, as $\mathcal{W}_{\mathcal{H}}^0(\alpha,0) \equiv \mathcal{W}_{\mathcal{H}}^0(\alpha)$ \;\;(see \cite{ghosh17}), $\mathcal{W}_\mathcal{H}^0(0,\beta) \equiv \mathcal{P}_\mathcal{H}^0(\beta)$ \;(see \cite{s13}), $\mathcal{W}_\mathcal{H}^0(1,0) \equiv \mathcal{W}_\mathcal{H}^0$ \;(see \cite{nagpal14}), and $\mathcal{W}_\mathcal{H}^0(0,0)\equiv \mathcal{P}_{\mathcal{H}}^0$ \;(see \cite{li13}).
 
In this article, we establish that functions in the class $\mathcal{W}_\mathcal{H}^0(\alpha,\beta)$ are close-to-convex $\mathbb{D}$. In section $3,$ we obtain certain coefficient inequalities and growth results for the functions in $\mathcal{W}_\mathcal{H}^0(\alpha,\beta)$. In section $4,$ we prove that the functions in $\mathcal{W}_\mathcal{H}^0(\alpha,\beta)$ are closed under convex combinations and establish certain convolution results. In section $5$, we determine the radius of convexity of partial sums $s_{p,q}(f)$ of functions in $\mathcal{W}_\mathcal{H}^0(\alpha,\beta)$. Finally, in section $6,$ we consider the harmonic mappings which involve the hypergeometric function and obtain conditions on its parameters such that it belongs to the class $\mathcal{W}_\mathcal{H}^0(\alpha,\beta).$ Further we construct the univalent harmonic polynomials belonging to $\mathcal{W}_\mathcal{H}^0(\alpha,\beta).$ The following results will be needed in our investigation.
\medskip
\noindent
\begin{lemma} \label{LEMA} \,(see, \cite{goodman}). Let $p\in \mathcal{P},$ where $\mathcal{P}$ denotes the class of Carath\'eodory functions in $\mathbb{D}.$ Then 
$$\left|p'(z)\right|\geq \dfrac{1-|z|}{1+|z|}\qquad {\rm and} \qquad \left|\dfrac{p''(z)}{p'(z)}\right|\leq \dfrac{2}{1-|z|^2} \quad (z\in \mathbb{D}).$$ 
These inequalities are sharp. Equality occurs for suitable $z\in \mathbb{D}$ if and only if $p(z)=-z-2e^{i \theta}\log (1-z e^{i \theta}) \quad (0\leq \theta\leq 2 \pi).$
\end{lemma}
  
\begin{lemma}[see \cite{clunie}]\label{lm.6}
If the harmonic mapping $f=h+\overline{g}:\mathbb{D}\rightarrow\mathbb{C}$ satisfies $|g'(0)|<|h'(0)|$ and the function  $F_\epsilon=h+\epsilon g$ is close-to-convex for every $|\epsilon|=1,$ then $f$ is close-to-convex function.
\end{lemma}
\section{The Close-to-Convexity}
\setcounter{equation}{0}

The first  result provides a one-to-one correspondence between the classes $\mathcal{W}_{\mathcal{H}}^0(\alpha, \beta)$ of harmonic mappings and the class $\mathcal{W}(\alpha, \beta)$ of analytic functions.

\begin{theorem} The harmonic mapping $f=h+\overline{g} \in \mathcal{W}_\mathcal{H}^0(\alpha,\beta)$ if and only if $F_{\epsilon}=h+\epsilon g \in \mathcal{W}(\alpha,\beta)$ for each $|\epsilon|=1$.
\label{th1}
\end{theorem}
\begin{proof}
Let $f=h+\overline{g} \in \mathcal{W}_\mathcal{H}^0(\alpha,\beta)$. Then for each $|\epsilon|=1$, we have
\begin{eqnarray*}
\Re(F_{\epsilon}'(z)+\alpha zF_{\epsilon}''(z))&=& \Re(h'(z)+\epsilon g'(z)+\alpha z(h''(z)+\epsilon g''(z))\\
&=& \Re(h'(z)+\alpha zh''(z)+\epsilon (g'(z)+\alpha g''(z))\\
&>& \Re(h'(z)+\alpha zh''(z))-|g'(z)+\alpha zg''(z)|> \beta \quad (z \in \mathbb{D}).
\end{eqnarray*}
Hence $F_{\epsilon} \in \mathcal{W}(\alpha,\beta)$ for each $|\epsilon|=1$. Conversely, let $F_{\epsilon} \in \mathcal{W}(\alpha,\beta).$ Then 
\begin{equation*}
\Re(h'(z)+\alpha zh''(z))> \Re(-\epsilon (g'(z)+\alpha zg''(z)))+\beta  \quad (z\in \mathbb{D}).
 \end{equation*}
As $\epsilon(|\epsilon|=1)$ is arbitrary, then for an appropriate choice of $\epsilon,$ we obtain
\begin{equation*}
\Re(h'(z)+\alpha zh''(z)-\beta) >|g'(z)+\alpha zg''(z)| \quad (z\in \mathbb{D}),
\end{equation*}
and hence we conclude that $f \in \mathcal{W}_\mathcal{H}^0(\alpha,\beta)$.
\end{proof}

To establish the next result, we need to establish that functions in the class $\mathcal{W}(\alpha,\beta)$ are close-to-convex in $\mathbb{D}$, and to prove this, we shall need the  following result.

\begin{lemma}   (Jack's Lemma \cite{jack71}) Let $\omega(z)$ be analytic in $\mathbb{D}$ with $\omega(0)=0.$ If $|\omega(z)|$ attains its maximum value on the circle $|z|=r<1$ at a point $z_0\in \mathbb{D}$, then we have $z_0 \omega'(z_0)=k\omega(z_0)$ for a real number $k\geq 1.$
\label{3}
\end{lemma}

\begin{lemma} If $f \in  \mathcal{W}(\alpha,\beta)$, then $\Re(f'(z))>\beta \,\,(0\leq\beta<1)$, and hence $f$ is close-to-convex in $\mathbb{D}$.
\label{5}
\end{lemma}	

\begin{proof}
If $f \in \mathcal{W}(\alpha,\beta),$ then $\Re (\psi(z))>0$, where $\psi(z)=f'(z)+\alpha z f''(z)-\beta$. Let $w$ be an analytic function in $\mathbb{D}$ such that $w(0)=0$ and  
\begin{equation*}
f'(z)=\frac{1+(1-2\beta)w(z)}{1-w(z)}.
\end{equation*} 
To prove the result, we need to show that $|w(z)|<1$ for all $z$ in $\mathbb{D}$. If not, then by Lemma \ref{3}, we could find some $\xi (|\xi|<1)$, such that $|w(\xi)|=1$ and $\xi w'(\xi)=kw(\xi)$, where $k \geq 1$. A computation gives
\begin{eqnarray*}
\Re \left\{\psi(\xi) \right\}&=&\Re \left\{\frac{1+(1-2\beta)w(\xi)}{1-w(\xi)}+\frac{2\alpha k(1-\beta)w(\xi)}{(1-w(\xi))^2}-\beta \right\} \\
&=& \Re \left\{\frac{ 2\alpha k(1-\beta)w(\xi)}{(1-w(\xi))^2} \right\} = - \frac{4\alpha k(1-\beta) (1-\Re(w(\xi))}{|1-w(\xi)|^4} \leq 0
\end{eqnarray*}
for $|w(\xi)|=1$. This contradicts the hypotheses. Hence, $|w(z)|<1,$ which lead to $\Re(f'(z))>\beta \,\,(0\leq\beta<1).$  
\end{proof}

\begin{theorem} \label{newthm}
The functions in the class $\mathcal{W}_{\mathcal{H}}^0(\alpha, \beta)$ are close-to-convex in $\mathbb{D}.$
\end{theorem}
\begin{proof} From Lemma \ref{5}, we find that functions $F_\epsilon=h+\epsilon g \in \mathcal{W}(\alpha, \beta)$ are close-to-convex in $\mathbb{D}$ for each $\epsilon(|\epsilon|=1).$ Now in view of Lemma \ref{lm.6} and Theorem \ref{th1}, we obtain that functions in $\mathcal{W}_{\mathcal{H}}^0(\alpha, \beta)$ are close-to-convex in $\mathbb{D}.$
\end{proof}

\section{Coefficient Inequalities and Growth Estimates}
\setcounter{equation}{0}

The following results provides sharp coefficient bounds for the functions in $\mathcal{W}_{\mathcal{H}}^0(\alpha, \beta).$ 

\begin{theorem} \label{thm1} Let $f=h+\overline{g}\in \mathcal{W}_\mathcal{H}^0(\alpha,\beta)$ be of the form \eqref{intro1} with $b_1=0.$ Then we have
\begin{equation}
|b_n|\leq\dfrac{1-\beta}{n(1+\alpha(n-1))}.
\label{eq6}
\end{equation}
The result is sharp and equality in \eqref{eq6} is obtained by $f(z)=z+\dfrac{1-\beta}{n(1+\alpha(n-1))}\overline{z}^n$.
\end{theorem}
\begin{proof}
Since $f=h+\overline{g} \in \mathcal{W}_\mathcal{H}^0(\alpha,\beta)$, then using the series representation of $g$, we have
\begin{eqnarray*}	
r^{n-1} n(1+\alpha(n-1))|b_n|&\leq& \frac{1}{2\pi}\int_{0}^{2\pi}|g'(re^{i\theta })+\alpha re^{i\theta}g''(re^{i\theta})|d\theta\\
&<& \frac{1}{2\pi}\int_{0}^{2\pi}\{\Re(h'(re^{i\theta})+\alpha re^{i\theta}h''(re^{i\theta}))-\beta\}d\theta \\
&=& \dfrac{1}{2 \pi} \int_0^{2\pi}\{1-\beta+ n(1+\alpha (n-1))a_n r^{n-1} e^{i(n-1)\theta} \}d \theta =1-\beta. 
\end{eqnarray*}
Now $r\rightarrow 1^{-}$ gives the desired bound. Further, it is easy to see that the equality in \eqref{eq6} is obtained for the function $f(z)=z+\dfrac{1-\beta}{n(1+\alpha(n-1))}\overline{z}^n$. 
\end{proof}

\begin{theorem}\label{th3}
Let  $f=h+\overline{g} \in \mathcal{W}_\mathcal{H}^0(\alpha,\beta)$ be of the form \eqref{intro1} with $b_1=0.$ Then for $n \geq 2$, we have
\begin{itemize}
\item[(i)] $|a_n|+|b_n|\leq \dfrac{2(1-\beta)}{n(1+\alpha(n-1))},$
\item[(ii)] $||a_n|-|b_n||\leq \dfrac{2(1-\beta)}{n(1+\alpha(n-1))},$
\item[(iii)] $|a_n|\leq \dfrac{2(1-\beta)}{n(1+\alpha(n-1))}. $
\end{itemize}
All these results are sharp for the function $f(z)=z+\sum_{n=2}^{\infty}\dfrac{2(1-\beta)}{n(1+\alpha(n-1))}\overline{z}^n$.
\end{theorem}
\begin{proof}
(i) Since $f=h+\overline{g} \in \mathcal{W}_\mathcal{H}^0(\alpha,\beta)$, then Theorem \ref{th1} implies that $F_{\epsilon}=h+\epsilon g \in \mathcal{W}(\alpha,\beta)$ for each $\epsilon(|\epsilon|=1)$. Thus for each $|\epsilon|=1$, we have
\begin{equation*}
\Re((h+\epsilon g)'(z)+\alpha z(h+\epsilon g)''(z))>\beta \quad {\rm{for}} \quad z \in \mathbb{D}.
\end{equation*}
This implies that there exists a {\it Carath\'{e}odory} function of the form $p(z)=1+\sum_{n=1}^{\infty}p_nz^n$, with $\Re(p(z))>0$ in $\mathbb{D}$, such that
\begin{equation}
h'(z)+\alpha zh''(z)+\epsilon (g'(z)+\alpha zg''(z))=\beta+(1-\beta)p(z).
\label{eq7}
\end{equation}
Comparing coefficients on both sides of \eqref{eq7}, we obtain
\begin{equation}
n(1+\alpha(n-1))(a_n+\epsilon b_n)=(1-\beta)p_{n-1}\quad {\rm for} \quad n\geq 2.
\label{eq8}
\end{equation} 
Since $|p_n|\leq 2$ for $n\geq 1$ (see \cite[p. 41]{Durenp}), and $\epsilon(|\epsilon|=1)$ is arbitrary, therefore the result follows from \eqref{eq8}. Part (ii) and (iii) follows from part (i).
\end{proof}

\medskip
The following result gives a sufficient condition for a function to be in the class $\mathcal{W}_\mathcal{H}^0(\alpha,\beta)$.

\begin{theorem}\label{th5}
Let $f=h+\overline{g} \in \mathcal{H}^0$, where $h$ and $g$ are of the form \eqref{intro1}. If
\begin{equation}
\sum_{n=2}^{\infty}n(1+\alpha(n-1))(|a_n|+|b_n|)\leq 1-\beta,
\label{eq13}
\end{equation}
then $f \in \mathcal{W}_\mathcal{H}^0(\alpha,\beta)$.
\end{theorem}
\begin{proof}
If $f=h+\bar{g} \in \mathcal{H}^0$, then using \eqref{eq13}, we have
\begin{eqnarray*}
\Re(h'(z)+\alpha zh''(z))&=& \Re\Big(1+\sum_{n=2}^{\infty} n(1+\alpha(n-1))\,a_n \,z^{n-1}\Big) \\
&\geq& 1-\sum_{n=2}^{\infty}n(1+\alpha(n-1))|a_n| \geq\sum_{n=2}^{\infty}n(1+\alpha(n-1))|b_n|+\beta\\
&\geq& |\sum_{n=2}^{\infty}n(1+\alpha(n-1)) \,b_n|+\beta =|g'(z)+\alpha zg''(z)|+\beta,
\end{eqnarray*}
and so $f \in \mathcal{W}_\mathcal{H}^0(\alpha,\beta)$.
\end{proof}

\medskip
The following theorem gives sharp inequalities in the class $\mathcal{B}_{\mathcal{H}}^0(\alpha, \beta).$

\begin{theorem} \label{th4} If $f=h+\overline{g} \in \mathcal{W}_\mathcal{H}^0(\alpha,\beta)$, then 
\begin{equation}
|z|-2\sum_{n=2}^{\infty}\dfrac{(-1)^{n-1}(1-\beta)|z|^n}{\alpha n^2+n(1-\alpha)} \leq |f(z)|\leq  |z|+2\sum_{n=2}^{\infty}\dfrac{(1-\beta)|z|^n}{\alpha n^2+n(1-\alpha)}.
\label{eq9} 
\end{equation}
Both the inequalities are sharp when $f(z)=z+\sum_{n=2}^{\infty} \dfrac{2(1-\beta)}{\alpha n^2+n(1-\alpha)}\overline{z}^n$, or its rotations.
\end{theorem}

\begin{proof}
Let $f=h+\bar{g} \in \mathcal{W}_\mathcal{H}^0(\alpha,\beta)$. Then $F_{\epsilon}=h+\epsilon g \in \mathcal{W}(\alpha,\beta)$ for each $\epsilon\,(|\epsilon|=1)$. Thus there exists an analytic function $w(z)$ with $w(0)=0$ and $|w(z)|<1$ in $\mathbb{D}$, such that
\begin{equation}
F_{\epsilon}'(z)+\alpha zF_{\epsilon}''(z)=\frac{1+(1-2\beta)w(z)}{1-w(z)}.
\label{eq11}
\end{equation}
Simplifying \eqref{eq11}, we get
\begin{eqnarray*}
z^{1/\alpha}F_{\epsilon}'(z) =\dfrac{1}{\alpha}\int_{0}^{z}\xi^{\frac{1}{\alpha}-1}\frac{1+(1-2\beta)w(\xi)}{1-w(\xi)}d\xi = \dfrac{1}{\alpha}\int_{0}^{|z|}(te^{i\theta})^{\frac{1}{\alpha}-1}\frac{1+(1-2\beta)w(te^{i\theta})}{1-w(te^{i\theta})}e^{i\theta}dt.
\end{eqnarray*}
Therefore using Schwarz Lemma, we have
\begin{eqnarray*}
|z^{1/\alpha}F_{\epsilon}'(z)|=\Big|\dfrac{1}{\alpha}\int_{0}^{|z|}(te^{i\theta})^{\frac{1}{\alpha}-1}\frac{1+(1-2\beta)w(te^{i\theta})}{1-w(te^{i\theta})}e^{i\theta}dt\Big| \leq \frac{1}{\alpha}\int_{0}^{|z|}t^{\frac{1}{\alpha}-1}\frac{1+(1-2\beta)t}{1-t}dt,
\end{eqnarray*}
and
\begin{eqnarray*}
|z^{1/\alpha}F_{\epsilon}'(z)|&=&\Big|\dfrac{1}{\alpha}\int_{0}^{|z|}(te^{i\theta})^{\frac{1}{\alpha}-1}\frac{1+(1-2\beta)w(te^{i\theta})}{1-w(te^{i\theta})}e^{i\theta}dt\Big|\\
&\geq &\dfrac{1}{\alpha}\int_{0}^{|z|}t^{\frac{1}{\alpha}-1}\; \Re{\frac{1+(1-2\beta)w(te^{i\theta})}{1-w(te^{i\theta})}}dt\\
&\geq&\frac{1}{\alpha}\int_{0}^{|z|}t^{\frac{1}{\alpha}-1}\; \frac{1+(1-2\beta)t}{1-t}dt.
\end{eqnarray*}
Further computation gives
\begin{equation}
|F'(z)|=|h'(z)+\epsilon g'(z)|\leq 1+2(1-\beta)\sum_{n=1}^{\infty}\frac{|z|^n}{1+\alpha n},
\label{eq12}
\end{equation}
and
\begin{equation*}
|F'(z)|=|h'(z)+\epsilon g'(z)|\geq 1+2(1-\beta)\sum_{n=1}^{\infty}\frac{(-1)^n|z|^n}{1+\alpha n}.
\end{equation*}
Since $\epsilon(|\epsilon|=1)$ is arbitrary, it follows from \eqref{eq12} that
\begin{equation*}
|h'(z)|+| g'(z)|\leq 1+2(1-\beta)\sum_{n=1}^{\infty}\frac{|z|^n}{1+\alpha n},
\end{equation*}
and
\begin{equation*}
|h'(z)|-| g'(z)|\geq 1-2(1-\beta)\sum_{n=1}^{\infty}\frac{(-1)^n|z|^n}{1+\alpha n}.
\end{equation*}
Let $\Gamma$ be the radial segment from 0 to $z$, then
\begin{eqnarray*}
|f(z)|&=&\Big|\int_\Gamma \dfrac{\partial f}{\partial \xi}d\xi +\frac{\partial f}{\partial \bar{\xi}}d\bar{\xi}\Big|\leq \int_\Gamma (|h'(\xi)|+| g'(\xi)|)|d\xi|\\
&\leq& \int_{0}^{|z|}\Big( 1+2(1-\beta)\sum_{n=1}^{\infty}\dfrac{|t|^n}{1+\alpha n}\Big)dt=|z|+2(1-\beta)\sum_{n=2}^{\infty}\frac{|z|^n}{\alpha n^2+(1-\alpha)n},
\end{eqnarray*}
and
\begin{eqnarray*}
|f(z)|&=&\int_\Gamma \Big|\dfrac{\partial f}{\partial \xi}d\xi +\frac{\partial f}{\partial \bar{\xi}}d\bar{\xi}\Big|\geq \int_\Gamma (|h'(\xi)|-| g'(\xi)|)|d\xi|\\
&\geq& \int_{0}^{|z|}\Big( 1-2(1-\beta
)\sum_{n=1}^{\infty}\frac{(-1)^n|t|^n}{1+\alpha n}\Big)dt=|z|+2(1-\beta)\sum_{n=2}^{\infty}\frac{(-1)^{n-1}|z|^n}{\alpha n^2+(1-\alpha)n}.
\end{eqnarray*}
Equality in \eqref{eq9} holds for the function $ f(z)=z+\sum_{n=2}^{\infty}\dfrac{2(1-\beta)}{\alpha n^2+(1-\alpha)n}\overline{ z}^n$
or its rotations.
\end{proof}
 
\section{Convex combinations and convolutions}
\setcounter{equation}{0}

In this section, we prove that the class $\mathcal{W}_\mathcal{H}^0(\alpha,\beta)$ is closed under convex combinations and convolutions. A sequence $\{c_n\}_{n=0}^{\infty}$ of non-negative  real numbers is said to be a convex null sequence, if $c_n\rightarrow 0$ as $n\rightarrow \infty$, and $c_0-c_1\geq c_1-c_2 \geq c_2-c_3 \geq...\geq c_{n-1}-c_n\geq ...\geq 0.$ To prove results for convolution, we shall need the following Lemma \ref{7} and \ref{8}.

\begin{lemma} \cite{LFlf} \label{7} If $\{c_n\}_{n=0}^{\infty}$ be a convex null sequence, then function $q(z)=\dfrac{c_0}{2}+\sum_{n=1}^{\infty}c_nz^n$
is analytic and $\Re(q(z)) >0$ in $\mathbb{D}$.	
\end{lemma}

\begin{lemma}\cite{singh89}\label{8}
Let the function $p$ be analytic in $\mathbb{D}$ with $p(0)=1$ and $\Re(p(z))>1/2$ in $\mathbb{D}$. Then for any analytic function $f$ in $\mathbb{D}$, the function $p*f$ takes values in the convex hull of the image of $\mathbb{D}$ under $f$.
\end{lemma}

\begin{theorem}\label{th6} The class $\mathcal{W}_\mathcal{H}^0(\alpha,\beta)$ is closed under convex combinations.
\end{theorem}
\begin{proof}
Let $f_i=h_i+\overline{g_i} \in \mathcal{W}_\mathcal{H}^0(\alpha,\beta)$ for $i=1,2,...n$ and $\sum_{i=1}^{n}t_i=1(0\leq t_i \leq 1)$. Write the convex combination of $f_i's$ as
\begin{equation*}
f(z)=\sum_{i=1}^{n}t_if_i(z)=h(z)+\overline{g(z)},
\end{equation*}
where $ h(z)=\sum_{i=1}^{n}t_ih_i(z)$ and $ g(z)=\sum_{i=1}^{n}t_ig_i(z)$. Clearly both $h$ and $g$ are analytic in $\mathbb{D}$ with $h(0)=g(0)=h'(0)-1=g'(0)=0.$ A simple computation yields
\begin{eqnarray*}
\Re(h'(z)+\alpha zh''(z))&=& \Re\Big(\sum_{i=1}^{n}t_i(h'(z)+\alpha zh''(z))\Big) > \sum_{i=1}^{n}t_i(|g_i'(z)+\alpha zg_i''(z)|+\beta)\\
&\geq& |g'(z)+\alpha zg''(z)|+\beta.
\end{eqnarray*}
This shows that $f \in \mathcal{W}_\mathcal{H}^0(\alpha,\beta)$.
\end{proof}

\begin{lemma}\label{9}
If $F \in \mathcal{W}(\alpha,\beta)$, then $\Re\Big(\dfrac{F(z)}{z}\Big) > \dfrac{1}{2-\beta}$.
\end{lemma}

\begin{proof}
If $F \in \mathcal{W}(\alpha,\beta)$ be given by $F(z)=z+\sum_{n=2}^{\infty}A_nz^n$, then
\begin{equation*}
\Re \Big(1+\sum_{n=2}^{\infty} n(1+\alpha (n-1))A_n z^{n-1}\Big)>\beta \quad (z \in \mathbb{D}),
\end{equation*}
which is equivalent to $\Re(p(z))>\dfrac{1}{2-\beta}\geq \dfrac{1}{2}\;$ in $\mathbb{D}$, where $p(z)=1+\dfrac{1}{2-\beta}\sum_{n=2}^{\infty}n(1+\alpha (n-1))A_n z^{n-1}.$ Now consider a sequence $\{c_n\}_{n=0}^{\infty}$ defined by $c_0=1$ and $c_{n-1}=\dfrac{2-\beta}{n(1+\alpha(n-1))}$ for $n\geq 2$. We can easily see that the sequence $\{c_n\}_{n=0}^{\infty}$ is convex null sequence and hence in view of Lemma \ref{7}, the function $q(z)=\frac{1}{2}+\sum_{n=2}^{\infty}\dfrac{2-\beta}{n(1+\alpha(n-1))}z^{n-1}$ is analytic and $\Re(q(z))>0$ in $\mathbb{D}$. Further 
$$\frac{F(z)}{z}=p(z)*\Big(1+\sum_{n=2}^{\infty}\dfrac{2-\beta}{n(1+\alpha(n-1))}z^{n-1}\Big).$$
Hence an application of Lemma \ref{8} gives that $\Re\Big(\dfrac{F(z)}{z}\Big)>\dfrac{1}{2-\beta}$ for $z\in \mathbb{D}$.
\end{proof}

\begin{lemma}\label{lm10} Let $F_1$ and $F_2$ belong to $\mathcal{W}(\alpha,\beta)$, then $F_1*F_2 \in \mathcal{W}(\alpha,\beta)$.
\end{lemma}
\begin{proof}
The convolution of $F_1=z+\sum_{n=2}^{\infty}A_nz^n$ and  $F_2=z+\sum_{n=2}^{\infty}B_nz^n$ is given by
$$ F(z)=(F_1*F_2)(z)=z+\sum_{n=2}^{\infty}A_nB_nz^n.$$
Since $zF'(z)=zF_1'(z)*F_2(z)$, therefore a computation shows that 
\begin{equation}
\frac{1}{1-\beta} \big(F'(z)+z\alpha F''(z)-\beta \big) =\frac{1}{1-\beta}(F_1'(z)+z\alpha F_1''(z)-\beta)*\Big(\frac{F_2(z)}{z}\Big).
\label{eq14}
\end{equation}
Since  $F_1 \in \mathcal{W}(\alpha,\beta)$, hence it satisfy $\Re(F_1'(z)+\alpha zF_1''(z)-\beta)>0.$ Further from Lemma \ref{9}, we have $\Re(\dfrac{F_2(z)}{z})> \dfrac{1}{2-\beta}\geq\dfrac{1}{2}$ in $\mathbb{D}$. Now applying Lemma \ref{8}, we get $F=F_1*F_2 \in \mathcal{W}(\alpha,\beta)$.	
\end{proof}

Now using Lemma \ref{lm10}, we will show that the class $\mathcal{W}_\mathcal{H}^0(\alpha,\beta)$ is closed under convolutions.

\begin{theorem}\label{11}
If functions $f_1$ and $f_2$ belong to $\mathcal{W}_\mathcal{H}^0(\alpha,\beta),$ then $f_1*f_2 \in \mathcal{W}_\mathcal{H}^0(\alpha,\beta)$.
\end{theorem} 
\begin{proof}
Let the functions $f_1=h_1+\overline{g_1}$ and $f_2=h_2+\overline{g_2}$ are belongs to $\mathcal{W}_\mathcal{H}^0(\alpha,\beta)$. To show $f_1*f_2 \in \mathcal{W}_\mathcal{H}^0(\alpha,\beta)$, it is sufficient to show that $F_{\epsilon}=h_1*h_2+\epsilon( {g_1*g_2}) \in \mathcal{W}(\alpha,\beta)$ for each $\epsilon(|\epsilon|=1)$. By Lemma \ref{lm10}, $\mathcal{W}(\alpha,\beta)$ is closed under convolutions. If $h_i+\epsilon g_i \in \mathcal{W}(\alpha,\beta)$ for each $\epsilon(|\epsilon|=1)$ and for $i=1,2$. Then both $F_1$ and $F_2$ given by
\begin{equation*}
F_1(z)=(h_1-g_1)*(h_2-\epsilon g_2) \quad {\mbox and} \quad  F_2(z)=(h_1+g_1)*(h_2+\epsilon g_2),
\end{equation*}
belong to $\mathcal{W}(\alpha,\beta)$. Since $\mathcal{W}(\alpha,\beta)$ is close under convex combinations, then the function $F_{\epsilon}=\dfrac{1}{2}(F_1+F_2)=(h_1*h_2)+\epsilon (g_1*g_2)$ belongs to $\mathcal{W}(\alpha,\beta)$. Hence $\mathcal{W}_\mathcal{H}^0(\alpha,\beta)$ is closed under convolution.
\end{proof} 

In \cite{Goodloe02}, Goodloe  considered the Hadamard product of a harmonic function with an analytic function defined as follows:
$$f\,\widehat{*}\phi=h*\phi+\overline{g*\phi},$$
where $f=h+\overline{g}$ is harmonic function and $\phi$ is an analytic function in $\mathbb{D}.$

\begin{theorem}\label{12}
\rm Let $f \in \mathcal{W}_{\mathcal{H}}^0(\alpha,\beta)$ and $\phi \in \mathcal{A}$ be such that $\Re\Big(\dfrac{\phi(z)}{z}\Big)> \dfrac{1}{2}$ for $z \in \mathbb{D}$, then $f\,\widehat{*}\phi \in \mathcal{W}_\mathcal{H}^0(\alpha,\beta)$.
\end{theorem}
\begin{proof}
Let $f=h+\overline{g} \in \mathcal{W}_\mathcal{H}^0(\alpha,\beta)$. To prove that $f\,\widehat{*}\phi$ belongs to $\mathcal{W}_\mathcal{H}^0(\alpha,\beta)$, it suffices to prove that $F_{\epsilon}=h*\phi +\epsilon (g*\phi)$ belongs to $\mathcal{W}(\alpha,\beta)$ for each $\epsilon (|\epsilon|=1)$. Since $f=h+\overline{g}\in \mathcal{W}_\mathcal{H}^0(\alpha,\beta),$ then $F_{\epsilon}=h+\epsilon g$ belongs to $\mathcal{W}(\alpha,\beta)$ for each $\epsilon (|\epsilon|=1)$. Therefore 
\begin{equation*}
\frac{1}{1-\beta} (F_\epsilon'(z)+\alpha zF_{\epsilon}''(z)-\beta)=\frac{1}{1-\beta}(F_{\epsilon}'(z)+\alpha zF_{\epsilon}''(z)-\beta)*\frac{\phi(z)}{z}.
\end{equation*}
Since $\Re\Big(\dfrac{\phi(z)}{z}\Big)> \dfrac{1}{2}$\, and\, $\Re(F_\epsilon'(z)+\alpha zF_\epsilon''(z))>\beta$ in $\mathbb{D}$, then in view of Lemma \ref{8}, we obtain that $F_\epsilon \in \mathcal{W}(\alpha,\beta)$.
\end{proof}

\begin{corollary}
\rm Suppose $f \in \mathcal{W}_\mathcal{H}^0(\alpha,\beta)$ and $\phi \in \mathcal{K}$, then $f\,\widehat*\phi \in \mathcal{W}_\mathcal{H}^0(\alpha,\beta)$.
\end{corollary}
\begin{proof}
It is well known that, if $\phi$ is convex then  $\Re\Big(\dfrac{\phi(z)}{z}\Big)> \dfrac{1}{2}$ for $z \in \mathbb{D}$. Hence result follows from Theorem \ref{12}.
\end{proof}

\section{Partial sums}
\setcounter{equation}{0}
In this section, we determine the value of $r$ such that the partial sums of $f \in \mathcal{W}_{\mathcal{H}}^0(\alpha, \beta)$ are convex in the disk $|z|<r.$
\begin{theorem}\label{P1}
Let $f=h+\overline{g}\in \mathcal{W}_{\mathcal{H}}^0(\alpha, \beta).$ If $p$ and $q$ satisfies one of the following conditions:
\begin{itemize}
\item[(i)] $1=p\,<\,q$
\item[(ii)] $3\,\leq \, p\,<\, q$
\item[(iii)] $3\leq q<p,$
\end{itemize}
then $s_{p,q}(f)(z)$ is convex in $|z|<1/4.$
\end{theorem}

\begin{proof} \;(i)\; By assumption, we know that
$$s_{1,q}(f)(z)=z+\overline{s_q(g)(z)}=z+\sum_{n=2}^q\overline{b_n z^n}.$$
Since
$$\Re\left\{\dfrac{z+\overline{z(zs'_q(g)(z))'}}{z-\overline{zs'_q(g)(z)}} \right\}=\Re \left\{ \dfrac{z+\sum_{n=2}^q \overline{n^2b_n z^n}}{z-\sum_{n=2}^q\overline{n b_n z^n}}\right\}\quad {\rm and} \quad \lim _{z \rightarrow 0} \dfrac{z+\sum_{n=2}^q\overline{n^2b_n z^n}}{z-\sum_{n=2}^q\overline{n b_n z^n}}=1,$$ 
it suffices to prove 
$$A=:\Re \left\{\left(z+\sum_{n=2}^q\overline{n^2 b_n z^n}\right) \left(\overline{z}-\sum_{n=2}^qn b_n z^n \right) \right\}>0 \qquad {\rm for}\qquad |z|=1/4.$$
Now, we find that
\begin{eqnarray*}
\qquad A &=& |z|^2+\Re \left(\sum_{n=2}^q \overline{n^2 b_n z^{n+1}}-\sum_{n=2}^q n b_n z^{n+1}\right)-\Re \left\{\left(\sum_{n=2}^q\overline{n^2 b_n z^n} \right) \left(\sum_{n=2}^q nb_n z^n \right) \right\}   \\
 &\geq& |z|^2-\sum_{n=2}^q n(n-1)|b_n||z|^{n+1}-\left(\sum_{n=2}^q n^2 |b_n||z|^n \right)\left(\sum_{n=2}^q n|b_n| |z|^n \right).
\end{eqnarray*}
Further, using Theorem \ref{thm1}, we obtain
\begin{eqnarray*}
A &\geq& |z|^2-\sum_{n=2}^q \dfrac{(1-\beta)(n-1)}{1+\alpha (n-1)} |z|^{n+1}-\left(\sum_{n=2}^{\infty}\dfrac{n (1-\beta)}{1+\alpha(n-1)}|z|^n \right) \left(\sum_{n=2}^{\infty}\dfrac{(1-\beta)}{1+\alpha (n-1)}|z|^n \right)\notag \\
&\geq& |z|^2-(1-\beta)\sum_{n=2}^q (n-1) |z|^{n+1}-(1-\beta)^2\left(\sum_{n=2}^{\infty}n |z|^n \right) \left(\sum_{n=2}^{\infty}|z|^n \right) \notag \\
&=& |z|^2-(1-\beta)|z|^3\dfrac{1-q|z|^{q-1}+(q-1)|z|^q}{(1-|z|)^2}\\
&& \qquad -(1-\beta)^2 |z|^4 \dfrac{(2-|z|-(q+1)|z|^{q-1}+q|z|^q)(1-|z|^{q-1})}{(1-|z|)^3}. \notag
\end{eqnarray*}
Thus, for $|z|=1/4$, we have
\begin{eqnarray*}
\dfrac{A\,(1-|z|)^3}{|z|^2} &\geq&(1-|z|)^3-(1-\beta)|z|(1-|z|)(1-q|z|^{q-1}+(q-1)|z|^q) \notag \\
&& \;\;\; -(1-\beta)^2|z|^2(2-|z|-(q+3)|z|^{q-1}+(q+1)|z|^q+(q+1)|z|^{2q-2}-q|z|^{2q-1}) \notag \\
&\geq& \dfrac{27}{64}-\dfrac{3}{16}\left(1-\dfrac{q}{4^{q-1}}-\dfrac{q-1}{4^q}\right)-\dfrac{1}{16} \left(\dfrac{7}{4}-\dfrac{q+3}{4^{q-1}}+\dfrac{q+1}{4^q}+\dfrac{q+1}{4^{2(q-1)}}-\dfrac{q}{4^{2q-1}}\right) \notag \\
&=& \dfrac{1}{8}+ \dfrac{12q+14}{4^{q+2}}-\dfrac{3q+4}{4^{2q-1}} = \dfrac{1}{8}+\dfrac{12 q(4^q-1)+14\times 4^q-16}{4^{2q+2}}>0. \notag
\end{eqnarray*}
Hence the result follows.

(ii) Let $\sigma_p(h)(z)=\sum_{n=p+1}^{\infty}a_n z^n$ and $\sigma_q(g)(z)=\sum_{n=q+1}^{\infty}b_nz^n,$ so that $h(z)=s_p(h)(z)+\sigma_p(h)(z)$ and $g(z)=s_q(g)(z)+\sigma_q(g)(z).$ Thus for each $|\epsilon|=1$, we may write
\begin{equation}\label{.eq1}
1+z\,\dfrac{s_p''(h)(z)+\epsilon s_q''(g)(z)}{s_p'(h)(z)+\epsilon s'_q(g)(z)}=1+\phi(z)+\psi(z),
\end{equation}
where
$$\phi(z)=\dfrac{z(h''(z)+\epsilon g''(z))}{h'(z)+\epsilon g'(z)}\qquad \rm and$$
$$\psi(z)=\dfrac{\phi(z)(\sigma'_p(h)(z)+\epsilon\sigma'_q(g)(z))-z(\sigma_p''(h)(z)+\epsilon\sigma_q''(g)(z))}{h'(z)+\epsilon g'(z)-(\sigma_p'(h)(z)+\epsilon \sigma'_q(g)(z))}.$$
Since $h+\epsilon g \in \mathcal{P},$ using  Lemma \ref{LEMA}, we have 
\begin{equation}\label{.eq2}
|\phi(z)|\leq \dfrac{2|z|}{1-|z|^2}\qquad {\rm and} \qquad |h'(z)+\epsilon g'(z)|\geq \dfrac{1-|z|}{1+|z|}.
\end{equation}
Now, if $p\leq q,$ then Theorem \ref{thm1},  yields that 
\begin{eqnarray}\label{.eq3}
|\sigma_p'(h)(z)+\epsilon \sigma'_q(g)(z)| &=& \left| \sum_{n=p+1}^qna_nz^{n-1}+\sum_{n=q+1}^\infty n(a_n+\epsilon b_n)z^{n-1} \right| \notag \\
&\leq & \sum_{n=p+1}^{\infty}\dfrac{2(1-\beta)}{1+\alpha(n-1)}|z|^{n-1}\,\leq\, 2(1-\beta)\sum _{n=p+1}^{\infty}|z|^{n-1} \notag \\
&=& 2(1-\beta) \dfrac{|z|^p}{1-|z|}.
\end{eqnarray}
Similarly,
\begin{eqnarray} \label{.eq4}
|z(\sigma_p''(h)(z)+\epsilon \sigma _q''(g)(z))| &= &\left|\sum_{n=p+1}^q n(n-1)a_nz^{n-1}+\sum_{n=q+1}^{\infty} n(n-1)(a_n+\epsilon b_n)z^{n-1} \right| \notag \\ 
&\leq& \sum_{n=p+1}^{\infty}\dfrac{2(1-\beta)(n-1)}{1+\alpha (n-1)}|z|^{n-1}\,\leq \,2(1-\beta)\sum_{n=p+1}^{\infty}(n-1)|z|^{n-1} \notag \\
&= & 2(1-\beta) \left(\dfrac{p|z|^p}{1-|z|}+\dfrac{|z|^{p+1}}{(1-|z|)^2} \right).
\end{eqnarray}
Using estimates \eqref{.eq2} - \eqref{.eq4}, by the triangle inequality we deduce that 
$$\left| \psi(z)\right|  \leq  \dfrac{2(1-\beta)|z|^p\{3|z|+|z|^2+p(1-|z|^2)\}}{(1-|z|)\{(1-|z|)^2-2(1-\beta)|z|^p(1+|z|)\}}.$$
Thus
\begin{eqnarray}
\Re(1+\phi(z)+\psi(z))&\geq& 1-|\phi(z)|-|\psi(z)|\notag \\
&\geq& 1-\dfrac{2|z|}{1-|z|^2}-\dfrac{2(1-\beta)|z|^p\{3|z|+|z|^2+p(1-|z|^2)\}}{(1-|z|)\{(1-|z|)^2-2(1-\beta)|z|^p(1+|z|)\}}\notag\\
&=& \dfrac{1-|z|^2-2|z|}{1-|z|^2}-\dfrac{2(1-\beta)|z|^p\{3|z|+|z|^2+p(1-|z|^2)\}}{(1-|z|)\{(1-|z|)^2-2(1-\beta)|z|^p(1+|z|)\}}, \notag
\end{eqnarray}
which for $|z|=1/4$ gives
$$\Re(1+\phi(z)+\psi(z))\geq \dfrac{1}{3}\left\{\dfrac{7}{5}-\dfrac{2(1-\beta)(13+15 p)}{9\times 4^{p-1}-10(1-\beta)} \right\}=B(p,\beta).$$
Since the function $B(p,\beta)$ is monotonically increasing with respect to $p$ for $p\geq3,$ the least estimate shows that $\Re(1+\phi(z)+\psi(z))\geq A(p)\geq A(3)>0.$ Thus \eqref{.eq1} implies for each $|\epsilon|=1,$ that the section $s_p(h)+\epsilon s_q(g)$ is convex in $|z|\leq 1/4$ for $3\leq p \leq q.$ As $\epsilon$ is arbitrary, this shows that $s_{p,q}(f)$ is convex in $|z|<1/4,$ for $3\leq p\leq q.$
\medskip

(iii) If $p>q,$ then using Theorem \ref{thm1}, we have
\begin{eqnarray}\label{.eq5}
\left| \sigma_p'(h)(z)+\epsilon \sigma_q'(g)(z)\right| &=& \left|\sum_{n=q+1}^p\epsilon n b_n z^{n-1}+\sum_{n=p+1}^{\infty}n(a_n+\epsilon b_n)z^{n-1} \right| \notag \\
&\leq& \sum_{n=q+1}^p\dfrac{1-\beta}{1+\alpha (n-1)}|z|^{n-1}+\sum_{n=p+1}^{\infty}\dfrac{2(1-\beta)}{1+\alpha(n-1)}|z|^{n-1} \notag \\
&\leq& (1-\beta)\left(\sum_{n=q+1}^p|z|^{n-1} +2\sum _{n=p+1}^{\infty}|z|^{n-1}\right)\,=\, \dfrac{(1-\beta)(|z|^p+|z|^q)}{1-|z|}, 
\end{eqnarray}
and
\begin{eqnarray} \label{.eq6}
\left| z(\sigma_p''(h)(z)+\epsilon  \sigma_q''(g)(z)) \right| &=& \left| \sum_{n=q+1}^p\epsilon n(n-1)b_n z^{n-1}+\sum_{n=p+1}^{\infty}n(n-1)(a_n+\epsilon b_n)z^{n-1}\right| \notag \\
&\leq& \sum_{n=q+1}^p\dfrac{(n-1)(1-\beta)}{1+\alpha(n-1)}|z|^{n-1}+\sum_{n=p+1}^{\infty}\dfrac{2(n-1)(1-\beta)}{1+\alpha(n-1)}|z|^{n-1} \notag \\
&\leq& (1-\beta)\left(\sum_{n=q+1}^{\infty}(n-1)|z|^{n-1}+\sum_{n=p+1}^{\infty}2(n-1)|z|^{n-1} \right)\notag\\
&=& \dfrac{(1-\beta)\{p|z|^p+q|z|^q-(p-1)|z|^{p+1}-(q-1)|z|^{q+1}\}}{(1-|z|)^2}. 
\end{eqnarray}
Using estimates \eqref{.eq2}, \eqref{.eq5} and \eqref{.eq6}, we obtain that
$$ |\psi(z)|\leq \dfrac{(1-\beta)}{(1-|z|)}\left(\dfrac{p|z|^p+q|z|^q+3|z|^{p+1}+3|z|^{q+1}-(p-1)|z|^{p+2}-(q-1)|z|^{q+2}}{1-2|z|+|z|^2-(1-\beta)(1+|z|)(|z|^p+|z|^q)} \right). $$
Thus $\Re \left( 1+\phi(z)+\psi(z)\right)\geq 1-|\phi(z)|-|\psi(z)|,$ which for $|z|=1/4$ reduces to 
\begin{eqnarray}
\Re \left( 1+\phi(z)+\psi(z)\right)&\geq &  \dfrac{4}{3}\left(\dfrac{7}{20}-\dfrac{(1-\beta)\{4^p(15 q+13)+4^q(15 p+13)\}}{9\times 4^{p+q}-20(1-\beta)(4^p+4^q)} \right) \notag\\
&>&   \dfrac{4}{3}\left(\dfrac{7}{20}-\dfrac{4^p(15 q+13)+4^q(15 p+13)}{9\times 4^{p+q}-20(4^p+4^q)} \right). \notag 
\end{eqnarray}
Moreover, for $p>q\geq3,$ we have
$$ \Re \left( 1+\phi(z)+\psi(z)\right)>  \dfrac{4}{3}\left(\dfrac{7}{20}-\dfrac{305}{2204}\right)>0, \notag$$
which implies that for each $\epsilon$ with $|\epsilon|=1,$ $s_p(h)+\epsilon s_q(g)$ is convex in $|z|<1/4,$ for $3\leq q\leq p,$ and thus each section $s_{p,q}(f)$ is convex in $|z|<1/4$ for $3\leq q \leq p.$ 
\end{proof} 

\begin{theorem}
Let $f=h+\overline{g}\in \mathcal{W}_{\mathcal{H}}^0(\alpha, \beta).$ Then
\begin{itemize}
\item[(i)] For $q>2,\, s_{2,q}(f)(z)$ is convex in the disk $|z|<R_1, $ where $R_1$ is smallest positive root of the equation
\begin{equation}\label{.eq7}
1-4r+(6\beta-2)r^2-8(1-\beta)r^3+(1-2\beta)r^4+4(1-\beta)r^5=0
\end{equation}
in $(0,1).$
\item[(ii)] For $p>2,\, s_{p,2}(f)(z)$ is convex in the disk $|z|<R_2,$ where $R_2$ is smallest positive root of the equation
\begin{equation}\label{.eq8}
1-4r+(1-\beta)r^2-(8-3\beta)r^3-(5-2\beta)r^4-(4-3\beta)r^5+3(1-\beta)r^6=0
\end{equation} 
in $(0,1).$
\end{itemize}
\end{theorem}
\begin{proof} (i) Let $f=h+\overline{g}\in \mathcal{W}_{\mathcal{H}}^0(\alpha, \beta),$ and suppose that $p=2<q.$ Then for each $|\epsilon|=1,$ it is sufficient to show that 
$$X=\Re\left(1+\dfrac{z(s_2''(h)(z)+\epsilon s_q''(g)(z))}{s_2'(h)(z)+\epsilon s'_q(g)(z)} \right)>0$$
in the disk $|z|<R_1.$ For $2=p<q,$ the estimates in \eqref{.eq2}-\eqref{.eq4} are continue to hold. Therefore, we deduce that
\begin{eqnarray}
(1-|z|)X &\geq & \dfrac{1-|z|^2-2|z|}{1+|z|}-\dfrac{2(1-\beta)|z|^2\{3|z|+2(1-|z|^2)+|z|^2\}}{1-2|z|+|z|^2-2(1-\beta)|z|^2(1+|z|)} \notag \\
&=&  \dfrac{1-|z|^2-2|z|}{1+|z|}-\dfrac{2(1-\beta)|z|^2\{2+3|z|-|z|^2\}}{1-2|z|+(1-2(1-\beta))|z|^2-2(1-\beta)|z|^3} \notag \\
&=& \dfrac{1-4|z|+\{4-6(1-\beta)\}|z|^2-8(1-\beta)|z|^3+\{2(1-\beta)-1\}|z|^4+4(1-\beta)|z|^5}{(1+|z|)\{1-2|z|+(1-2(1-\beta))|z|^2-2(1-\beta)|z|^3\}} \notag \\
&=& \dfrac{1-4|z|+(6\beta-2)|z|^2-8(1-\beta)|z|^3+(1-2\beta)|z|^4+4(1-\beta)|z|^5}{(1+|z|)\{1-2|z|+(1-2(1-\beta))|z|^2-2(1-\beta)|z|^3\}}, \notag
\end{eqnarray}
which is greater then zero in $|z|<R_1,$ where $R_1$ is the smallest positive root of the equation \eqref{.eq7} in $(0,1).$

(ii) Let $f=h+\overline{g}\in \mathcal{W}_{\mathcal{H}}^0(\alpha, \beta)$ and suppose that $q=2<p.$ Then for each $|\epsilon|=1,$ it is sufficient to show that 
$$Y=\Re\left(1+\dfrac{z(s_p''(h)(z)+\epsilon s_2''(g)(z))}{s_p'(h)(z)+\epsilon s'_2(g)(z)} \right)>0$$
in the disk $|z|<R_2.$ Since for $2=q<p,$ the estimates in equations \eqref{.eq2}, \eqref{.eq5} and \eqref{.eq6} continue to hold. Therefore we deduce that 
\begin{eqnarray}
(1-|z|)Y &\geq & \dfrac{1-|z|^2-2|z|}{1+|z|}-\dfrac{(1-\beta)(p|z|^p+2|z|^2+3|z|^{p+1}+3|z|^4-(p-1)|z|^{p+2}-|z|^4)}{1-2|z|+|z|^2-(1-\beta)(|z|^p+|z|^2)(1+|z|)} \notag \\
&=& \dfrac{1-|z|^2-2|z|}{1+|z|}-\dfrac{(1-\beta)|z|^p(p+3|z|-(p-1)|z|^2)+(1-\beta)|z|^2(2+3|z|-|z|^2)}{1-2|z|+|z|^2-(1-\beta)(|z|^p+|z|^2)(1+|z|)} \notag \\
&\geq & \dfrac{1-|z|^2-2|z|}{1+|z|}-\dfrac{(1-\beta)\{|z|^3(3+3|z|-2|z|^2)+|z|^2(2+3|z|-|z|^2)\}}{1-2|z|+|z|^2-(1-\beta)(2|z|^3+|z|^2+|z|^4)} \notag \\
&=& \dfrac{1-|z|^2-2|z|}{1+|z|}-\dfrac{(1-\beta)\{3|z|^3+3|z|^4-2|z|^5+2|z|^2+3|z|^3-|z|^4\}}{1-2|z|+|z|^2-(1-\beta)(2|z|^3+|z|^2+|z|^4)} \notag \\
&=&\dfrac{1-4|z|+(1+\beta)|z|^2-(8-3\beta)|z|^3-(5-2\beta)|z|^4-(4-3\beta)|z|^5+3(1-\beta)|z|^6}{(1+|z|)\{1-2|z|+|z|^2-(1-\beta)(|z|^2+2|z|^3+|z|^4)\}},\notag
\end{eqnarray}
which is greater then zero in $|z|<R_2,$ where $R_2$ is the smallest positive root of \eqref{.eq8} in $(0,1).$
\end{proof}

\begin{theorem} If $f=h+\overline{g}\in \mathcal{W}_{\mathcal{H}}^0(\alpha, \beta),$ then $s_{2,2}(f)(z)$ is convex in $|z|<(1+\alpha)/4(1-\beta).$
\end{theorem}
\begin{proof} Let $s_{2,2}(f)(z)\in \mathcal{W}_{\mathcal{H}}^0(\alpha, \beta).$ Then for each $|\epsilon|=1$, it is sufficient to show that 
$$\Re\left(1+\dfrac{z(s_2''(h)(z)+\epsilon s_2''(g)(z)}{s_2'(h)(z)+\epsilon s_2'(g)(z)} \right)>0$$
in the disk $|z|<\dfrac{1+\alpha}{4(1-\beta)}.$ In the view of Theorem \ref{th3}, we have
\begin{eqnarray}
\Re\left(1+\dfrac{z(s_2''(h)(z)+\epsilon s_2''(g)(z)}{s_2'(h)(z)+\epsilon s_2'(g)(z)} \right) &\geq& 1-\left|\dfrac{z(s_2''(h)(z)+\epsilon s_2''(g)(z)}{s_2'(h)(z)+\epsilon s_2'(g)(z)} \right| \notag \\
&=& 1-\left|\dfrac{2(a_2+\epsilon b_2)z}{1+2(a_2+\epsilon b_2)z} \right| \geq 1-\dfrac{2|a_2+\epsilon b_2||z|}{1-2|a_2+\epsilon b_2| |z|} \notag \\
&=& \dfrac{1-\dfrac{4(1-\beta)}{1+\alpha}|z|}{1-\dfrac{2(1-\beta)}{1+\alpha}|z|}>0 \notag.
\end{eqnarray}
Hence the result follows.
\end{proof}


\section{Applications}
\setcounter{equation}{0}

In this section, we consider the harmonic mappings whose co-analytic part involve the Gaussian hypergeometric function $_2F_1(a,b;c;z)$, which is defined by
\begin{equation}\label{G1}
_2F_1(a,b;c;z)=F(a,b;c;z)=\sum_{n=0}^{\infty}\dfrac{(a)_n \,(b)_n}{(c)_n\, n!}z^n \qquad (z\in \mathbb{D}),
\end{equation}
where $a,b,c \in \mathbb{C}, c\neq 0, -1, -2, \cdots$ and $(a)_n$ is the Pochhammer symbol defined by $(a)_n=a(a+1)(a+2)\cdots(a+n-1)$ and $(a)_0=1$ for $n\in \mathbb{N}.$ The series \eqref{G1} is absolutely convergent in $\mathbb{D}.$ Moreover, if $\Re(c-a-b)>0,$ then the series \eqref{G1} is convergent in $|z|\leq 1.$ Further, for $z=1,$ we have the following well-known Gauss formula \cite{NMT}
\begin{equation}\label{G2}
F(a,b;c;1)=\dfrac{\Gamma(c)\Gamma(c-a-b)}{\Gamma(c-a)\Gamma(c-b)}<\infty.
\end{equation}

We shall use the following Lemma to prove our results in this section:
\begin{lemma}\label{lemaG} \cite{G21} \ Let $a,b>0.$ Then the following holds:
\begin{itemize}
\item[(i)] For $c>a+b+1,$
$$\sum_{n=0}^{\infty}\dfrac{(n+1)(a)_n (b)_n}{(c)_n n!}=\dfrac{\Gamma(c) \Gamma(c-a-b-1)}{\Gamma(c-a)\Gamma(c-b)}(ab+c-a-b-1).$$
\item[(ii)] For $c>a+b+2,$
$$\sum_{n=0}^{\infty} \dfrac{(n+1)^2(a)_n (b)_n}{(c)_n n!}=\dfrac{\Gamma(c)\Gamma(c-a-b)}{\Gamma(c-a) \Gamma(c-b)}\left(\dfrac{(a)_2 (b)_2}{(c-a-b-2)_2}+\dfrac{3ab}{c-a-b-1}+1 \right).$$
\item[(iii)] For $a,b,c\neq1$ with $c>\,{\rm max}\, \{0, a+b+1\},$
$$\sum_{n=0}^{\infty}\dfrac{(a)_n (b)_n}{(c)_n (n+1)!}=\dfrac{1}{(a-1)(b-1)}\left[\dfrac{\Gamma(c)\Gamma(c-a-b-1)}{\Gamma(c-a)\Gamma(c-b)}-(c-1) \right].$$
\end{itemize}
\end{lemma}

\begin{theorem} \label{thmG}
Let $f_1(z)=z+\overline{ z^2 F(a,b;c;z)},\quad f_2(z)=z+\overline{z(F(a,b;c;z)-1)}$ and $f_3(z)=z+\overline{z\int_0^z F(a,b;c;t)dt},$ where $a,b,c$ are positive real numbers such that $c>a+b+2.$ Then the following holds:
\begin{itemize}
\item[(i)] If
\begin{equation}\label{G3}
\dfrac{\Gamma(c) \Gamma(c-a-b-1)}{\Gamma(c-a) \Gamma(c-b)}\left[  \dfrac{\alpha (a)_2 (b)_2}{c-a-b-2}+(1+4\alpha)ab+2(1+\alpha)(c-a-b-1)\right]\leq 1-\beta,
\end{equation}
then $f_1\in \mathcal{W}_{\mathcal{H}}^0(\alpha, \beta).$
\item[(ii)] If
\begin{equation}\label{G4}
\dfrac{\Gamma(c)\Gamma(c-a-b-2)}{\Gamma(c-a)(c-b)}\left[\alpha ab(ab+c-1)+(1+\alpha)ab(c-a-b-2)+1 \right] \leq 2-\beta,
\end{equation}
then $f_2 \in \mathcal{W}_{\mathcal{H}}^0(\alpha, \beta).$
\item[(iii)] If $a,b,c\neq 1 $ and $c>\,{\rm max}\, \{0, a+b+1\},$
\begin{equation}\label{G5}
\dfrac{\Gamma(c)\Gamma(c-a-b-1)}{\Gamma(c-a)\Gamma(c-b)}\left[\alpha ab+(1+2\alpha)(c-a-b-1)+\dfrac{1}{(a-1)(b-1)} \right]
\end{equation}
$$ -\dfrac{(c-1)}{(a-1)(b-1)} \leq 1-\beta,$$
then $f_3 \in \mathcal{W}_{\mathcal{H}}^0(\alpha, \beta).$
\end{itemize}
\end{theorem}

\begin{proof}
(i) Let $f_1(z)=z+\overline{ z^2 F(a,b;c;z)}=z+\overline{\sum_{n=2}^{\infty}C_nz^n,}$ where
$$C_n=\dfrac{(a)_{n-2}(b)_{n-2}}{(c)_{n-2}(n-2)!} \quad {\rm for}\quad n\geq 2.$$
Therefore, we have 
\begin{eqnarray}  
\sum_{n=2}^{\infty}n(1+\alpha(n-1))|C_n| &=& \sum_{n=2}^{\infty}n(1+\alpha(n-1)) \dfrac{(a)_{n-2}(b)_{n-2}}{(c)_{n-2}(n-2)!}\notag \\
&=& (1+\alpha)\sum_{n=0}^{\infty}(n+1)\dfrac{(a)_n(b)_n}{(c)_n n!}+\alpha \sum_{n=0}^{\infty}(n+1)^2\dfrac{(a)_n (b)_n}{(c)_n n!}+\sum_{n=0}^{\infty}\dfrac{(a)_n(b)_n}{(c)_n n!}. \notag
\end{eqnarray}
Now, using Lemma \ref{lemaG} and Gauss formula \eqref{G2}, we have\\
$\sum_{n=2}^{\infty}n(1+\alpha(n-1))|C_n|=$
$$\dfrac{\Gamma(c) \Gamma(c-a-b-1)}{\Gamma(c-a) \Gamma(c-b)}\left[ \alpha \dfrac{(a)_2 (b)_2}{c-a-b-2}+(1+4\alpha)ab+2(1+\alpha)(c-a-b-1)\right].$$
If \eqref{G3} holds, then $\sum_{n=2}^{\infty}n(1+\alpha(n-1))|C_n|\leq 1-\beta.$ Hence the result follows.
\medskip

(ii) Let $f_2(z)=z+\overline{z(F(a,b;c;z)-1)}=z+\overline{\sum_{n=2}^{\infty}D_nz^n},$ where
$$D_n=\dfrac{(a)_{n-1}(b)_{n-1}}{(c)_{n-1}(n-1)!}\quad {\rm for}\quad n\geq2.$$
Therefore, we have
\begin{eqnarray} \label{G7}
\sum_{n=2}^{\infty}n(1+\alpha(n-1))|D_n| &=& \sum_{n=2}^{\infty}n(1+\alpha(n-1))\dfrac{(a)_{n-1}(b)_{n-1}}{(c)_{n-1}(n-1)!} \notag \\
&=& \sum_{n=0}^{\infty}(\alpha(n+1)^2+(1+\alpha)(n+1)+1)\dfrac{(a)_{n+1}(b)_{n+1}}{(c)_{n+1}(n+1)!}. \notag
\end{eqnarray}
Now using the identity $(\gamma)_{n+1}=\gamma(\gamma+1)_n$, we have
\begin{eqnarray}\label{G8}
\sum_{n=2}^{\infty}n(1+\alpha(n-1))|D_n|&=& \dfrac{ab}{c}  \alpha \sum_{n=0}^{\infty}(n+1)\dfrac{(a+1)_n(b+1)_n}{(c+1)_n n!} \notag \\
&+& \dfrac{ab}{c}\left[(1+\alpha)\sum_{n=0}^{\infty} \dfrac{(a+1)_n(b+1)_n}{(c+1)_n n!}+\sum_{n=0}^{\infty} \dfrac{(a+1)_n(b+1)_n}{(c+1)_n (n+1)!} \right]. \notag 
\end{eqnarray}
Further, using Lemma \ref{lemaG} and Gauss formula \eqref{G2}, we obtain \\
$\sum_{n=2}^{\infty}n(1+\alpha(n-1))|D_n|= $
$$\dfrac{\Gamma(c) \Gamma(c-a-b-2)}{\Gamma(c-a) \Gamma(c-b)}\left[\alpha ab(ab+c-1)+(1+\alpha)ab(c-a-b-2)+1 \right]-1.$$
Now, if \eqref{G4} holds, then $\sum_{n=2}^{\infty}n(1+\alpha(n-1))|D_n| \leq 1-\beta,$ hence the result follows.

\medskip

(iii) Let $f_3(z)=z+\overline{z\int_0^z F(a,b;c;t)dt}=z+\overline{\sum_{n=2}^{\infty}E_nz^n},$ where
$$E_n=\dfrac{(a)_{n-2}(b)_{n-2}}{(c)_{n-2}(n-1)!}\quad {\rm for}\quad n\geq2.$$
Therefore,  
\begin{eqnarray} \label{G9}
\sum_{n=2}^{\infty}n(1+\alpha(n-1))|E_n| &=& \sum_{n=2}^{\infty}n(1+\alpha(n-1))\dfrac{(a)_{n-2}(b)_{n-2}}{(c)_{n-2}(n-1)!} \notag \\
&=& \alpha \sum_{n=0} ^{\infty} (n+1) \dfrac{(a)_n(b)_n}{(c)_n n!}+(1+\alpha) \sum_{n=0} ^{\infty}\dfrac{(a)_n(b)_n}{(c)_n n!}+\sum_{n=0} ^{\infty} \dfrac{(a)_n(b)_n}{(c)_n (n+1)!}. \notag
\end{eqnarray}
Now using Lemma \ref{lemaG} and Gauss formula \eqref{G2}, we obtain
\begin{equation*}\label{10}
\sum_{n=2}^{\infty}n(1+\alpha(n-1))|E_n|=
\end{equation*}
 $$\dfrac{\Gamma(c)\Gamma(c-a-b-1)}{\Gamma(c-a)\Gamma(c-b)}\left[\alpha ab+(1+2\alpha)(c-a-b-1)+\dfrac{1}{(a-1)(b-1)}\right]- \dfrac{(c-1)}{(a-1)(b-1)}.$$
Further, if \eqref{G5} holds, then  the result follows.
\end{proof}

\medskip
Note that for $\eta \in \mathbb{C}/ \{-1, -2, \cdots \}$ and $n\in \mathbb{N}\cup \{0\},$ we have
$$\dfrac{(-1)^n(-\eta)_n}{n!}=\binom {\eta}{n} = \dfrac{\Gamma(\eta+1)}{n! \Gamma(\eta-n+1)}.$$
In particular, when $\eta=m (m\in \mathbb{N}, m\geq n), $ we have
$$(-m)_n=\dfrac{(-1)^n m!}{(m-n)!}.$$
Using above relations in Theorem \ref{thmG}, we get harmonic univalent polynomials which belongs to the class $\mathcal{W}_{\mathcal{H}}^0(\alpha, \beta).$ Setting $a=b=-m\,(m\in \mathbb{N}),$ we get

\begin{corollary}\label{c11} Let $m\in \mathbb{N},$ $c$ be a positive real number and 
$$F_1(z)=z+\overline{\sum_{n=0}^m \binom {m}{n} \dfrac{(m-n+1)_n}{(c)_n}z^{n+2}},$$
$$F_2(z)=z+\overline{\sum_{n=0}^m \binom {m}{n} \dfrac{(m-n+1)_n}{(c)_n}z^{n+1}},$$
$$F_3(z)=z+\overline{\sum_{n=0}^m \binom {m}{n} \dfrac{(m-n+1)_n}{(n+1)(c)_n}z^{n+2}}.$$
Then the following holds:
\begin{itemize}
\item[(i)] If
\begin{equation}
\dfrac{\Gamma(c) \Gamma(c+2m-1)}{[\Gamma(c+m)]^2}\left[  \dfrac{\alpha m^2 (m-1)^2}{c+2m-2}+(1+4\alpha)m^2+2(1+\alpha)(c+2m-1)\right]\leq 1-\beta,
\end{equation}
then $F_1\in \mathcal{W}_{\mathcal{H}}^0(\alpha, \beta).$
\item[(ii)] If
\begin{equation} 
\dfrac{\Gamma(c)\Gamma(c+2(m-1))}{[\Gamma(c+m)]^2}\left[\alpha m^2(m^2+c-1)+(1+\alpha)m^2(c+2m-2)+1 \right] \leq 2-\beta,
\end{equation}
then $F_2 \in \mathcal{W}_{\mathcal{H}}^0(\alpha, \beta).$
\item[(iii)] If
\begin{equation} 
\dfrac{\Gamma(c)\Gamma(c+2m-1)}{[\Gamma(c+m)]^2}\left[\alpha m^2+(1+2\alpha)(c+2m-1)+\dfrac{1}{(m+1)^2}  \right]-\dfrac{(c-1)}{(m+1)^2}\leq 1-\beta,
\end{equation}
then $F_3 \in \mathcal{W}_{\mathcal{H}}^0(\alpha, \beta).$
\end{itemize}
\end{corollary} 

Further setting $m=2$ and $c=1$ in Corollary \ref{c11}, we get
 
\begin{corollary} If $G_1(z)=z+\overline{z^2+4z^3+z^4},\,\,G_2(z)=z+\overline{4z^2+z^3},\,\,$ and $G_3 (z)=z+\overline{z^2+2z^3+\dfrac{1}{3}z^4},$ then the following holds:
\begin{itemize}
\item[(i)] If $2(19\alpha+9)\leq 1-\beta,$ then $G_1(z)\in\mathcal{W}_{\mathcal{H}}^0(\alpha, \beta).$
\item[(ii)] If $ 28\alpha+13\leq 2(2-\beta),$ then $G_2(z)\in\mathcal{W}_{\mathcal{H}}^0(\alpha, \beta).$
\item[(iii)] If $ 108\alpha+37\leq 6(1-\beta),$ then $G_3(z)\in\mathcal{W}_{\mathcal{H}}^0(\alpha, \beta).$
\end{itemize}
\end{corollary}

\end{document}